\documentclass[12pt]{article}
\usepackage[T2A]{fontenc}
\usepackage[utf8]{inputenc}
\usepackage{amssymb,amsmath,amsfonts,amsthm,amscd,latexsym,verbatim,graphics,epsfig,indentfirst,xcolor,array}
\usepackage{geometry}
\theoremstyle{plain}
\newtheorem{theorem}{Theorem}[section]
\newtheorem{lemma}[theorem]{Lemma}
\newtheorem{corollary}[theorem]{Corollary}

\theoremstyle{remark}
\newtheorem{example}[theorem]{Example}

\usepackage{hyperref}

\begin{document}

\begin{center}
{\Large The K\"othe Radical in Novikov Algebras}

\smallskip

Alexander S. Panasenko
\end{center}

\begin{abstract}
We study the K\"othe radical (the upper nil radical) in the variety of Novikov algebras. We prove that the class of weakly right nilpotent Novikov algebras coincides with the class of r-nil algebras and forms a radical class. We establish that the K\"othe radical is a hereditary radical in the variety of Novikov algebras. Furthermore, we obtain a structural description of the K\"othe radical as the intersection of ideals whose corresponding quotient algebras are prime $\mathcal{K}$-semisimple algebras. Finally, we construct an example of a prime r-nil Novikov algebra over a field of characteristic $p>0$, which implies that the Baer radical is strictly contained in the K\"othe radical in this case.

{\it Keywords}:
Novikov algebra, K\"othe radical, upper nil radical, Baer radical, weakly right nilpotent algebra, weakly solvable algebra, r-nil algebra.
\end{abstract}

\section{Introduction}

An algebra $A$ is called a \textit{Novikov algebra} if it satisfies the following identities:
\[(x,y,z)=(y,x,z), \qquad (xy)z=(xz)y,\]
where $(x,y,z)=(xy)z-x(yz)$ is the associator. The first identity expresses left symmetry, while the second expresses right commutativity. 

Novikov algebras first appeared in the work of I.~M.~Gel'fand and I.~Ya.~Dorfman~\cite{GD1979} in the context of Hamiltonian operators in formal variational calculus. Independently, they arose in the works of A.~A.~Balinsky and S.~P.~Novikov~\cite{BN1985} (see also~\cite{N1985,B1987}) in the study of local translation-invariant Lie algebras.

The structural theory of finite-dimensional Novikov algebras was initiated by E.~I.~Zelmanov~\cite{Z1987}, who described simple and semisimple finite-dimensional Novikov algebras over an algebraically closed field of characteristic $0$. It turned out that all such algebras are commutative. Later, simple finite-dimensional Novikov algebras over an algebraically closed field of characteristic $p>0$ were described by J.~M.~Osborn~\cite{Osborn1992} and X.~Xu~\cite{Xu1996} (see also~\cite{ZZ2024}). A complete classification of simple and semisimple finite-dimensional Novikov algebras over an arbitrary field of characteristic $p>0$ was recently obtained by V.~N.~Zhelyabin and A.~P.~Pozhidaev~\cite{PZ2026}.

The radical theory of Novikov algebras has been actively developing in recent years by the author. In~\cite{P2022}, the Baer radical of Novikov algebras was constructed and studied. In~\cite{P2024}, the Andrunakievich radical was described, and in~\cite{P2026}, the radicals of Lie-solvable Novikov algebras were investigated. Recently, in \cite{P2027}, the Jacobson radical was constructed and described.

The present paper is devoted to the study of the K\"othe radical (the upper nil radical) in the variety of Novikov algebras. The main tool of our investigation is the notion of a weakly right nilpotent algebra, introduced by A.~I.~Shirshov in~\cite{S1957} and investigated by I.~M.~Mikheev in~\cite{M1972} for right alternative algebras. The closely related notion of a weakly solvable algebra appeared in the 1960s in Lie theory and the corresponding radical was constructed by V.~A.~Parfenov~\cite{P1971}. The central result of the paper is the existence of the K\"othe radical in Novikov algebras. We show that the class of weakly right nilpotent Novikov algebras coincides with the class of r-nil Novikov algebras. This equivalence allows us to establish the radicality and heredity of the K\"othe radical. Moreover, we construct an example of a prime r-nil Novikov algebra over a field of positive characteristic, demonstrating that the K\"othe radical strictly contains the Baer radical. It is the main difference from commutative or Lie-solvable case \cite{P2026}.

The paper is organized as follows. Section~2 collects the necessary definitions and preliminary facts. In Section~3, we prove that weakly right nilpotent Novikov algebras form a radical class. In Section~4, we establish the coincidence of the weakly right nilpotent radical with the K\"othe radical, prove its heredity, and obtain a structural description of $\mathcal{K}$-semisimple algebras. In Section~5, we construct an example of a prime r-nil Novikov algebra over a field of positive characteristic.

\section{Preliminaries}

Let $\mathcal{C}$ be a class of algebras. We say that an algebra is a $\mathcal{C}$-algebra if it belongs to $\mathcal{C}$. We say that an ideal is a $\mathcal{C}$-ideal if it (considered as an algebra) belongs to the class $\mathcal{C}$.

Let $\mathcal{N}$ be a class of algebras closed under taking homomorphic images and ideals. A subclass $\mathcal{R}$ of $\mathcal{N}$ is called \textit{radical} in $\mathcal{N}$ if the following conditions are satisfied:
\begin{enumerate}
    \item The class $\mathcal{R}$ is closed under homomorphic images.
    \item Every algebra $A$ in $\mathcal{N}$ contains an $\mathcal{R}$-ideal $\mathcal{R}(A)$ that contains all $\mathcal{R}$-ideals of $A$.
    \item For every algebra $A$ in $\mathcal{N}$, the algebra $A/\mathcal{R}(A)$ contains no nonzero $\mathcal{R}$-ideals.
\end{enumerate}

\noindent If $\mathcal{R}$ is a radical, then an algebra is called $\mathcal{R}$\textit{-semisimple} if it contains no nonzero $\mathcal{R}$-ideals.

 A radical class $\mathcal{R}$ is called \textit{hereditary} in the class $\mathcal{N}$ if $\mathcal{R}(I)=I\cap\mathcal{R}(A)$ for every algebra $A$ in $\mathcal{N}$ and every ideal $I\subseteq A$.

 In an algebra, we denote the operator of right multiplication by an element $x$ by $R_x$:
\[R_x(a):=ax\]
and the left-normed product of the elements $x_1,\dots, x_n$ by $x_1\dots x_n$:
\[x_1\dots x_n := x_1R_{x_2}\dots R_{x_n}.\]
For example, $x_1x_2x_3:=(x_1x_2)x_3$.

Let $A$ be an algebra generated by $x_1,\dots,x_n$. The algebra $A$ is called \textit{right nilpotent with respect to the generating set} $F=\{x_1,\dots,x_n\}$ if there exists $t\in\mathbb{N}$ such that $x_{i_1}\dots x_{i_t}=0$ for all $i_k\in\{1,\dots,n\}$.

We combine the terminology from the papers~\cite{M1972} and~\cite{P1971}.

An algebra $A$ is called \textit{locally right nilpotent in the sense of Shirshov} (or \textit{weakly right nilpotent}) if every finite subset of its elements generates a subalgebra that is right nilpotent with respect to that subset.

We define a non-associative polynomial $S_n(x_1,\dots,x_{2^n})$ recursively as follows. 
\[S_0(x_1) = x_1, \quad S_n(x_1,\dots,x_{2^n})=S_{n-1}(x_1,\dots,x_{2^{n-1}})S_{n-1}(x_{2^{n-1}+1},\dots,x_{2^n}).\]

The algebra $A$ is called \textit{solvable with respect to the generating set} $F=\{x_1,\dots,x_n\}$ if there exists $t\in\mathbb{N}$ such that $S_t(x_{i_1},\dots,x_{i_{2^t}})=0$ for all $i_k\in\{1,\dots,n\}$.

An algebra $A$ is called \textit{locally solvable in the sense of Shirshov} (or \textit{weakly solvable}) if every finite subset of its elements generates a subalgebra that is solvable with respect to that subset.

I.~M.~Mikheev proved in~\cite{M1972} that weakly right nilpotent algebras form a radical subclass of the class of right alternative algebras. In the paper by V.~A.~Parfenov~\cite{P1971}, it was proved that weakly solvable algebras form a radical class (all statements in~\cite{P1971} are given for Lie algebras, but the proofs work for the class of all nonassociative algebras).

The class of weakly right nilpotent algebras is denoted by $\mathcal{M}$. If this class is radical in the variety under consideration, then this radical is called the \textit{weakly right nilpotent radical}.

Let $x$ be an element of a nonassociative algebra. We recursively define
\[x^1 = x, \quad x^n = x^{n-1}x.\]
We say that $x$ is \textit{r-nilpotent} if $x^n=0$ for some $n>0$. An algebra is called \textit{r-nil} if all of its elements are r-nilpotent. The class of all r-nil algebras is denoted by $\mathcal{K}$. If this class is radical in the variety under consideration, then this radical is called the \textit{upper nil radical} or the \textit{K\"othe radical}.

Recall that an ideal $I$ of an algebra is called \textit{trivial} if it has zero multiplication, $I^2=0$. Let us construct an ascending chain of ideals in an algebra $A$. Let $\mathcal{B}_0(A) = 0$, and let $\mathcal{B}_1(A)$ be the sum of all trivial ideals of $A$. Next, we define $\mathcal{B}_\alpha(A)$ recursively: if $\alpha$ is a limit ordinal, set $\mathcal{B}_{\alpha}(A)=\bigcup_{\beta<\alpha}\mathcal{B}_{\beta}(A)$; if $\alpha$ is not a limit ordinal, then $\mathcal{B}_{\alpha}(A)$ is defined as the full preimage in $A$ of the ideal $\mathcal{B}_1(A/\mathcal{B}_{\alpha - 1}(A))$ of the algebra $A/\mathcal{B}_{\alpha - 1}(A)$ under the canonical homomorphism:
\[\mathcal{B}_\alpha(A)/\mathcal{B}_{\alpha - 1}(A)=\mathcal{B}_1(A/\mathcal{B}_{\alpha - 1}(A)).\]

The constructed \textit{Baer chain} terminates at some step. The largest ideal of the chain is denoted by $\mathcal{B}(A)$. The class of Novikov algebras that coincide with the maximal ideal of their Baer chain is radical. The corresponding radical is called the \textit{Baer radical} and is denoted by $\mathcal{B}$. This radical is hereditary.

Let $A$ be an associative and commutative algebra with a derivation $d:A\to A$ and multiplication $(x,y)\to x\cdot y$. Then a space $A$ with a multiplication
\[x\circ y = x\cdot d(y)\]
is a Novikov algebra. This algebra is called \textit{Gelfand-Dorfman construction} of an algebra $A$ with a derivation $d$ and is denoted by $GD(A,d)$.

\section{Weakly Right Nilpotent Radical of Novikov Algebras}

In any Novikov algebra, the following identity holds (see \cite{UZ2021}):
\begin{equation}\label{eq}
a(bx_1\dots x_t)=abx_1x_2\dots x_t - \sum\limits_{i=1}^t (a,b,x_i)x_1\dots x_{i-1}x_{i+1}\dots x_t.
\end{equation}
This identity is a key technical tool for this section, as it allows us to rearrange factors in left-normed words.

\begin{lemma}\label{lem3.1}
Let $A$ be a Novikov algebra and $x_1,\dots,x_{s}\in A$. Then for any $0<t<s-1$ the element $x_1\dots x_s$ can be written as a sum of elements of the form
\[(x_{i_1}\dots x_{i_{s-t}})(x_{j_1}\dots x_{j_t})\]
and
\[(x_{i_1}\dots x_{i_{s-t-1}})(x_{j_1}\dots x_{j_{t+1}}).\]
\end{lemma}

\begin{proof}
Let $y=x_{1}\dots x_{s-t-1}$ and $y_k=x_{s-t}\dots x_{k-1}x_{k+1}\dots x_{s}$ for $s-t\le k\le s$. Then the identity \eqref{eq} implies
\[x_{1}\dots x_{s} = y(x_{s-t}\dots x_{s}) + \sum\limits_{k=s-t+1}^{s} (y,x_{s-t},x_k)x_{s-t+1}\dots x_{k-1}x_{k+1}\dots x_{s} =\]
\[= y(x_{s-t}\dots x_{s}) + \sum\limits_{k=s-t+1}^{s} (y,y_k,x_k) = y(x_{s-t}\dots x_{s}) + \sum\limits_{k=s-t+1}^{s} ((yx_k)y_k-y(y_kx_k)).\]
This proves the lemma.
\end{proof}

The following lemma is a strengthening of Lemma~\ref{lem3.1} and allows us to represent long words via products of bounded degree.

\begin{lemma}\label{lem3.2}
Let $A$ be a Novikov algebra and $x_1,\dots,x_{s}\in A$. Then for any $0<t<s-1$ the element $x_1\dots x_s$ can be written as a sum of elements of the form
\[
aR_{b_1}\dots R_{b_q}
\]
for some $q\ge 0$, where $a$ is a left-normed word of length from $1$ to $t+1$ in the generators $x_k$ and $b_i$ are left-normed words of length from $t$ to $t+1$ in the generators $x_k$.
\end{lemma}

\begin{proof}
We proceed by induction on $s > 0$. If $s \le t+1$, there is nothing to prove. Suppose $s > t+1$. Then, by Lemma 3.1, we can express $x_1 \dots x_s$ as a sum of elements of the form
\[(x_{i_1} \dots x_{i_{s-t}})(x_{j_1} \dots x_{j_t})\]
and
\[(x_{i_1} \dots x_{i_{s-t-1}})(x_{j_1} \dots x_{j_{t+1}}).\]
Let $y$ denote the first factor in such a term, and $z$ the second. Since the length of $y$ is less than $s$, the induction hypothesis implies that $y$ can be expressed as a sum of elements of the form
\[aR_{b_1} \dots R_{b_q},\]
where $a$ is a left-normed word of length between $1$ and $t+1$ in the generators $x_k$, and each $b_i$ is a left-normed word of length between $t$ and $t+1$ in the generators $x_k$. Setting $b_{q+1} = z$, we obtain
\[x_1 \dots x_s = aR_{b_1} \dots R_{b_q}R_{b_{q+1}}.\]
This proves the lemma.
\end{proof}

\begin{lemma}\label{lem3.25}
Let $A$ be a Novikov algebra and $x_1,\dots,x_{s}\in A$. Then for any $0<t$, $s\ge 2t$ the element $x_1\dots x_s$ can be written as a sum of elements of the form
\[
aR_{b_1}\dots R_{b_q}
\]
for some $q\ge 0$, where $a$ is a left-normed word of length from $t$ to $2t$ in the generators $x_k$ and $b_i$ are left-normed words of length from $t$ to $t+1$ in the generators $x_k$.
\end{lemma}

\begin{proof}
   We proceed by induction on $s \ge 2t$. If $s = 2t$, there is nothing to prove. Suppose $s > 2t$. Then, by Lemma~\ref{lem3.1}, we can express $x_1 \dots x_s$ as a sum of elements of the form
\[(x_{i_1} \dots x_{i_{s-t}})(x_{j_1} \dots x_{j_t})\]
and
\[(x_{i_1} \dots x_{i_{s-t-1}})(x_{j_1} \dots x_{j_{t+1}}).\]
Let $y$ denote the first factor in such a term, and $z$ the second. If $t\le \deg y \le 2t$ then we take $a=y$ and $q=1$. If $\deg y \ge 2t$ then the induction hypothesis implies that $y$ can be expressed as a sum of elements of the form
\[aR_{b_1} \dots R_{b_q},\]
where $a$ is a left-normed word of length between $t$ and $2t$ in the generators $x_k$, and each $b_i$ is a left-normed word of length between $t$ and $t+1$ in the generators $x_k$. Setting $b_{q+1} = z$, we obtain
\[x_1 \dots x_s = aR_{b_1} \dots R_{b_q}R_{b_{q+1}}.\]

This proves the lemma.
\end{proof}

\begin{lemma}\label{lem3.3}
    If $A$ is a Novikov algebra right nilpotent of index $t$ with respect to its set of generators $\{x_1,\dots,x_n\}$, then 
    \[R_{x_{i_2}}\dots R_{x_{i_t}}(A)=0\] 
    for all $i_k\in\{1,\dots,n\}$. 
\end{lemma}

\begin{proof}
    The statement of the lemma means that for any non-associative monomial 
    \[f(t_1,\dots,t_n)\] 
    in variables $t_1,\dots,t_n$ we have $f(x_1,\dots,x_n)R_{x_{i_2}}\dots R_{x_{i_t}}=0$. Let us prove it by induction on $\deg(f)$. If $\deg(f)=1$ then we can assume that $f=x_s$ and 
\[f(x_1,\dots,x_n)R_{x_{i_2}}\dots R_{x_{i_t}}=x_sx_{i_2}\dots x_{i_t} =0.\]
Suppose that we proved the statement for $\deg f < m$. Let $\deg f = m$. Then $f=f_1f_2$, where $\deg f_1 < m$. Then
\[f(x_1,\dots,x_n)R_{x_{i_2}}\dots R_{x_{i_t}}= f_1R_{f_2}R_{x_{i_2}}\dots R_{x_{i_t}} = f_1R_{x_{i_2}}\dots R_{x_{i_t}}R_{f_2} = 0\]
by the induction hypothesis. This proves the lemma.
\end{proof}

Lemma~\ref{lem3.25} is the main tool for proving that the class of weakly right nilpotent algebras is closed under extensions.

\begin{lemma}\label{lem3.4}
Let $A$ be a Novikov algebra and $I$ an ideal of $A$ such that the algebras $A/I$ and $I$ are weakly right nilpotent. Then the algebra $A$ is weakly right nilpotent.
\end{lemma}

\begin{proof}
Let $M=\{1,\dots,n\}$. Let $F=\{x_1,\dots,x_n\}\subseteq A$, let $\overline{F}$ be the image of $F$ under the canonical homomorphism $A\to A/I$, and suppose that the right nilpotency index of the subalgebra generated by $\overline{F}$ with respect to $\overline{F}$ is $t$. Then $x_{i_1}\dots x_{i_t} \in I$ for all $i_j\in M$.

By Lemma~\ref{lem3.25}, we can represent a left-normed word of length $s$ in the generators $x_k$ in the form $aR_{b_1}\dots R_{b_q}$, where $2t\ge \deg a\ge t$ and the $b_i$ are left-normed words of length from $t$ to $t+1$ in the generators $x_k$. Consider the set of all left-normed words $x_{i_1}\dots x_{i_l}$ with $t\le l\le 2t$. This is a finite set of elements of $I$; hence the subalgebra $B$ generated by this set is right nilpotent with respect to this set. Let $N$ be its right nilpotency index. The length of $aR_{b_1}\dots R_{b_q}$ in generators $\{x_1,\dots,x_n\}$ is $L\le 2t+q(t+1)\le (q+2)(t+1)$. If $q\ge N$ then $aR_{b_1}\dots R_{b_q}=0$. Then, by Lemma~\ref{lem3.3}, the subalgebra generated by the set $F$ in the algebra $A$ is right nilpotent with respect to $F$ of index at most $(t+1)(N+2)$. This proves the lemma.
\end{proof}

We are now in a position to prove the main result of this section.

\begin{theorem}
In the class of Novikov algebras, weakly right nilpotent algebras form a radical class.
\end{theorem}

\begin{proof}
Since a homomorphic image of a weakly right nilpotent algebra is a weakly right nilpotent algebra, it follows from Lemma~\ref{lem3.4} and the homomorphism theorem that the sum of finitely many weakly right nilpotent ideals is a weakly right nilpotent ideal. Locality allows one to pass to an infinite sum; that is, the sum $\mathcal{M}(A)$ of all weakly right nilpotent ideals of the algebra $A$ is a weakly right nilpotent ideal. It follows from Lemma~\ref{lem3.4} that the algebra $A/\mathcal{M}(A)$ contains no weakly right nilpotent ideals. This proves the theorem.
\end{proof}

\section{K\"othe Radical of Novikov Algebras}

In this section, we establish the connection between the weakly right nilpotent radical and the K\"othe radical. Recall that in any Novikov algebra the following identity holds (\cite{P2022}):
\[(xt,y,z)=(x,yt,z)=(x,y,z)t.\]
Right commutativity of Novikov algebras allows us to prove a key lemma relating the r-nilpotency of elements to the nilpotency of right multiplication operators.

\begin{lemma}\label{lem4.1}
Let $x$ be an r-nilpotent element of a Novikov algebra $A$ of index $m$. Then $R_x$ is a nilpotent operator on the algebra $A$ of index at most $m+1$.
\end{lemma}

\begin{proof}
Let $x^m=0$ and $y\in A$. We prove that $(yx^k)R_x^{t}=(yx^{k+1})R_x^{t-1}$ whenever $k+t-1\ge m$. Indeed,
\begin{multline*}
(yx^k)R_{x}^t = (yx^k)x R_x^{t-1} = (y,x^k,x)R_x^{t-1} + (yx^{k+1})R_x^{t-1} =\\ (y,x^{k+t-1},x) + (yx^{k+1})R_x^{t-1} = (yx^{k+1})R_x^{t-1}.
\end{multline*}
Then, by induction, $yR_x^{m+1}=(yx)R_{x}^{m} = (yx^{m})x = 0$. This proves the lemma.
\end{proof}

Lemma~\ref{lem4.1} allows us to establish the central result of the paper—the equivalence between weak right nilpotency and r-nilness for Novikov algebras.

\begin{theorem}\label{t4.2}
A Novikov algebra $A$ is weakly right nilpotent if and only if it is an r-nil algebra.
\end{theorem}

\begin{proof}
Let $A$ be a weakly right nilpotent algebra and $x\in A$. Then the subalgebra generated by the element $x$ is right nilpotent with respect to $\{x\}$ of some index $t$, that is, $x^t=f_t(x,\dots,x)=0$.

Conversely, let $A$ be an r-nil Novikov algebra and $F=\{x_1,\dots,x_n\}\subseteq A$. Denote the r-nilpotency indices of the elements $x_1,\dots,x_n$ by $m_1,\dots,m_n$, respectively. Let $m=m_1+\dots+m_n+n$. Then in the monomial $z=x_{i_1}\dots x_{i_m}$, at least one $x_i$ occurs $m_i+1$ times. If this $x_i$ is in the first position, then $f_m(x_{i_1},\dots,x_{i_m})=0$. Suppose that $x_i$ is not in the first position. Then, by right commutativity, the word $z$ can be rewritten so that it contains the subword $x_{i_1}R_{x_i}^{m_i+1}$, which is zero by Lemma~\ref{lem4.1}. Thus the subalgebra generated by $F$ is right nilpotent with respect to $F$ of index $m$. This proves the theorem.
\end{proof}

\begin{corollary}\label{c4.3}
In the variety of Novikov algebras, the class of r-nil algebras is radical.
\end{corollary}

Corollary~\ref{c4.3} immediately implies the following relationship between radicals.

\begin{corollary}\label{c4.4}
The Baer radical of a Novikov algebra is an r-nil ideal. Thus, for every Novikov algebra $A$, one has the inclusions
\[\mathcal{B}(A)\subseteq \mathcal{K}(A) = \mathcal{M}(A).\]
\end{corollary}

Next, we prove the heredity of the K\"othe radical.

\begin{theorem}
In the class of Novikov algebras, the K\"othe radical is hereditary.
\end{theorem}

\begin{proof}
It is well known that, to prove heredity, it suffices to consider the cases $\mathcal{M}(A)=0$ and $\mathcal{M}(A)=A$. Note that the case $\mathcal{M}(A)=A$ is obvious from the definition of the weakly right nilpotent radical.

Suppose that $\mathcal{M}(A)=0$ and $I$ is an ideal of the algebra $A$. Since every right nilpotent ideal is weakly right nilpotent, the $\mathcal{M}$-semisimple algebra $I/\mathcal{M}(I)$ contains no nonzero ideals with zero multiplication. This means that the algebra $I/\mathcal{M}(I)$ is semiprime. By the results of~\cite{P2022} (Lemma 6), $\mathcal{M}(I)$ is an ideal of the algebra $A$. Since $\mathcal{M}(I)$ is an $\mathcal{M}$-radical algebra, we have $\mathcal{M}(I)=0$. This proves the theorem.
\end{proof}

To obtain a structural description of the K\"othe radical, we need the following lemma.

\begin{lemma}\label{lem4.6}
Let $A$ be a Novikov algebra, and let $y\in A$ be an element that is not r-nilpotent. Suppose moreover that every nonzero ideal of $A$ contains $y^t$ for some integer $t>0$. Then $\mathcal{K}(A)=0$.
\end{lemma}

\begin{proof}
Let $K$ be a nonzero r-nil ideal of the algebra $A$. Then there exists $t>0$ such that $y^{t+r}\in K$ for all $r>0$. By Lemma~\ref{lem4.1}, the operators $R_{y^{t}},R_{y^{t+1}}$ are nilpotent. Their nilpotency indices are bounded by some number $Q$. Then, by Lemma~\ref{lem3.2}, $y^{2(t+1)Q}$ can be written as a sum of elements of the form $aR_{b_1}\dots R_{b_{r-1}}$, where $b_i=y^{t_i}$ for $t\le t_i\le t+1$, $\deg a\le t+1$ and $r>0$. We have 
\[2(t+1)Q \le t+1 + \sum t_i \le (t+1)r,\] 
so $r\ge 2Q$. There are at least $r-1\ge 2Q-1$ numbers $t_i$, each equal to $t$ or $t+1$. Hence at least $Q$ of them are equal. It follows that, by right commutativity, we may assume that there is a subword of the form $R_{b_i}^{Q}$. Therefore $y^{2(t+1)Q}=0$, a contradiction. This proves the lemma.
\end{proof}

An algebra is called \textit{prime} if, for any ideals $I,J$, the condition $IJ=0$ implies $I=0$ or $J=0$.

We will need the following statement.

\begin{lemma}\label{lem4.7}\textup{\cite{P2026}}
    Let $A$ be a Novikov algebra and let $x\in A$ be such that 
    \[(x^n)^2 = (x^{n+1})^2 = 0\] 
    for some $n>0$. Then $x^{2n+2} = 0$.
\end{lemma}

Now we can obtain the structural description of the K\"othe radical.

\begin{theorem}
Let $A$ be a Novikov algebra. Then its K\"othe radical $\mathcal{K}(A)$ is equal to the intersection of all ideals $I$ of $A$ such that the algebra $A/I$ is a prime $\mathcal{K}$-semisimple algebra.
\end{theorem}

\begin{proof}
Let $L$ denote the intersection of all ideals $I_{\alpha}$ of the algebra $A$, $\alpha\in T$, such that the algebra $A/I_{\alpha}$ is a prime $\mathcal{K}$-semisimple algebra. Since $\mathcal{K}(A)$ is the smallest ideal of $A$ such that the algebra $A/\mathcal{K}(A)$ is $\mathcal{K}$-semisimple, we have $\mathcal{K}(A)\subseteq I_{\alpha}$ for every $\alpha\in T$. Hence $\mathcal{K}(A)\subseteq L$.

Conversely, let $x\notin\mathcal{K}(A)$. This means that the ideal $J$ generated by the element $x$ is not weakly right nilpotent. By Theorem~\ref{t4.2}, it follows that the ideal $J$ contains an element $y$ that is not r-nilpotent. Consider the family $S$ of ideals of the algebra $A$ whose intersection with the set $\{y^k\mid k\in\mathbb{N}\}$ is empty. This family is nonempty and inductive; hence, by Zorn's lemma, $S$ contains a maximal element $I$. Let $J_1$ and $J_2$ be ideals of the algebra $A$ such that $J_1J_2\subseteq I$. Suppose that $I\subsetneq J_1$ and $I\subsetneq J_2$. Then there exists an integer $m>0$ such that $y^m\in J_1\cap J_2$. But then $y^{m+1}\in J_1\cap J_2$, and $(y^m)^2\in I$, $(y^{m+1})^2\in I$. By Lemma~\ref{lem4.7} $y^{2m+2}\in I$, a contradiction. Thus the algebra $A/I$ is prime. By construction, every nonzero ideal of the algebra $A/I$ contains some power of the element $y+I$, and $y+I$ is not an r-nilpotent element. By Lemma~\ref{lem4.6}, $\mathcal{K}(A/I)=0$. Since $y\notin I$, we have $x\notin I$. Therefore $L\subseteq\mathcal{K}(A)$. This proves the theorem.
\end{proof}

\begin{corollary}
A $\mathcal{K}$-semisimple Novikov algebra is a subdirect product of prime $\mathcal{K}$-semisimple Novikov algebras.
\end{corollary}

\section{An Example of a Prime r-nil Novikov Algebra}

In this section, we construct an example of a prime r-nil Novikov algebra over a field of positive characteristic. This example will show that the inclusion of the Baer radical into the K\"othe radical is strict.

\begin{example}
    Let $F$ be a field of characteristic $p>0$, $B = F[x_1,x_2,\dots]$ be an associative and commutative polynomial ring in countably many variables with a unit, $I=(x_1^p,x_2^p,\dots)$ and $A=B/I$. We can define a derivation $D:B\to B$ by $D(x_i)=x_{i+1}$ and $D(fg)=D(f)g+fD(g)$. It is easy to see that $D(I)\subseteq I$, so we can consider an induced derivation $d:A\to A$. We denote the algebra $GD(A,d)$ by $\mathrm{Sk}_p$.
\end{example}

For $x\in\mathrm{Sk}_p$ we use a notation 
\[T_{x}(a)=a\circ x.\]

\begin{lemma}\label{lem5.2}
Let $J$ be an ideal of $\mathrm{Sk}_p$ and let $w=x_{1}^{t_1}\dots x_{n}^{t_n}\in J$. Then for any $s>0$ there exist $q(s)\ge s$ and $0\le k_{j}\le p-1$ such that \[t=x_s^{k_s}\dots x_{q(s)}^{k_{q(s)}}\in J\] 
and $k_{q(s)}>0$.
\end{lemma}

\begin{proof}
We will prove the statement by induction on $s$. If $s=1$ then the statement is proved. Let $s>1$. Then the induction step allows us to assume that there is an element $t=x_{s-1}^{r_1}\dots x_{s+n-2}^{r_n}\in J$. We will prove the statement by the second induction on $r_1$. If $r_1=0$ then we are done. Let $r_1>0$. We have 
\[w_1=wT_{x_{s-1}}^{p-1-r_2}\dots T_{x_{s+n-3}}^{p-1-r_n}\in J.\] 
Let us notice that $w_1=x_{s-1}^{r_1}x_s^{p-1}\dots x_{s+n-2}^{p-1}$. Then $d(w_1)=1\circ w_1 \in J$. Let us notice that 
\[d(w_1) = x_{s-1}^{r_1}x_s^{p-1}\dots x_{s+n-2}^{p-2}x_{s+n-1}.\]
Then
\[w_2=d(w_1)T_{x_{s+n-2}}^{p-2}=x_{s-1}^{r_1}x_s^{p-1}\dots x_{s+n-3}^{p-1}x_{s+n-2}^{p-2} x_{s+n-1}^{p-1}\in J.\] 
Let us prove by induction that
\[x_{s-1}^{r_1}x_s^{p-1}\dots x_{s+n-k-3}^{p-1}x_{s+n-k-2}^{p-2}x_{s+n-k-1}^{p-1}\dots x_{s+n-1+k}^{p-1}\in J.\]
Indeed, if
\[w_3=x_{s-1}^{r_1}x_{s}^{p-1}\dots x_{s+n-k-2}^{p-1}x_{s+n-k-1}^{p-2}x_{s+n-k}^{p-1}\dots x_{s+n+k-1}^{p-1}\in J\]
then
\[w_4 = d(w_3)T_{x_{s+n+k-1}}^{p-2} \in J.\]
It is enough to notice that $w_4=x_{s-1}^{r_1}x_s^{p-1}\dots x_{s+n-k-3}^{p-1}x_{s+n-k-2}^{p-2}x_{s+n-k-1}^{p-1}\dots x_{s+n+k-1}^{p-1}$. If $k=n-1$ then
\[w_5=x_{s-1}^{r_1-1}x_s^{p-1}\dots x_{s+2n-2}^{p-1}\in J.\]
Using the induction on $r_1$ we can get $r_1=0$.
This proves the lemma.
\end{proof}

\begin{lemma}\label{lem5.25}
Let $I$ be an ideal in $\mathrm{Sk}_p$ and $0\neq x\in I$. Then $I$ contains an element $u=x_1^{p-1}x_2^{p-1}\dots x_k^{p-1}\in I$.
\end{lemma}

\begin{proof}
Let $I$ be an ideal in $\mathrm{Sk}_p$ and $0\neq x\in I$. Then there exists $N\ge 0$ such that 
\[x=\sum\limits_{i\in M}\alpha_i x_1^{k_{1,i}}\dots x_N^{k_{N,i}},\] 
where $0\le k_{m,i}<p$ for all $1\le m\le N$ and $i\in M$ and $k_{N,i}\neq 0$ for some $i\in M$ and $\alpha_i\neq 0$ for any $i\in M$. 

If $N=0$ then $1\in I$ and $1\circ x_1 = x_2\in I$. 

If $N=1$ then 
\[g=\sum\limits_{i=0}^{p-1}\alpha_i x_1^i\in I\] 
for some $\alpha_i\in F$. Then $g\circ x_1 = \sum\limits_{i=0}^{p-1}\alpha_i x_1^ix_2\in I$. So, we can assume that $N\ge 2.$

Define a number $r_2=\min\{k_{2,i}\mid i\in M\}$. Then an element $w_2=xT_{x_1}^{p-1-r_2}$ can be represented in the following way: $w_2=x_2^{p-1}h_2$, where $h_2$ does not contain $x_2$. We define a number $r_{3}=\min \{k_{3,i}\mid k_{2,i}=r_2\}.$ Then an element $w_3=w_2T_{x_2}^{p-1-r_3}$ can be represented in the following way: $w_3=x_2^{p-1}x_3^{p-1}h_3$, where $h_3$ does not contain $x_3$. We will recursively repeat this process by $w_{k+1}=w_{k}T_{x_{k}}^{p-1-r_{k+1}}$ and $r_{m}=\min\{k_{m,i}\mid k_{2,i}=r_2,\dots,k_{m-1,i}=r_{m-1}\}$. Then $w_N=x_2^{p-1}\dots x_N^{p-1}h_N$, where $h_N=\sum\limits_{j=0}^{p-1} \alpha_j x_1^{j}$ and $\alpha_j\neq 0$ for some $p-1\ge j\ge 0$. If $r_1=\min\{j\mid \alpha_j\neq 0\}$ then
\[w_{N+1}=x_1^{p-1-r_1}\circ w_N =  \alpha_{r_1}(p-1)x_1^{p-1}\dots x_{N-1}^{p-1}x_N^{p-2}x_{N+1}\in I.\]
It follows that
\[w_{N+2} = \frac{1}{(p-1)\alpha_{r_1}}w_{N+1} T_{x_{N}}^{p-2}T_{x_{N-1}} = x_1^{p-1}\dots x_{N+1}^{p-1}\in I.\]
This proves the lemma.
\end{proof}

\begin{theorem}\label{t5.3}
A Novikov algebra $\mathrm{Sk}_p$ satisfies the following statements:
\begin{enumerate}
    \item $\mathrm{Sk}_p$ is generated by $p$ elements;
    \item $\mathrm{Sk}_p$ is a prime algebra;
    \item $\mathrm{Sk}_p$ is an r-nil algebra of bounded index $p+1$;
    \item $\mathrm{Sk}_p$ is a weakly solvable algebra.
\end{enumerate}
\end{theorem}

\begin{proof}
The fact that $\mathrm{Sk}_p$ is an r-nil algebra of bounded index was proved in~\cite{P2026}. 

The algebra $\mathrm{Sk}_p$ is generated by a set $M=\{1,x_1,\dots, x_1^{p-1}\}$. Indeed, let $N$ be a subalgebra, generated by $M$.  Let us prove by induction on degree of $g=x_1^{k_1}\dots x_m^{k_m}\in\mathrm{Sk}_p$ that $g\in N$, $k_m\ge 1$. If $g=x_m$ then $x_m=1\circ x_{m-1}$ and induction on $m$ proves that $x_m\in N$ since $x_1\in N$. If $g=x_1^k$ then $g\in M\subseteq N$. So, we can assume that $m\ge 2$. We have $g_1=x_1^{k_1}\dots x_m^{k_m-1} \in N$, so $g=g_1\circ x_{m-1}\in N$.

Let us prove that it is prime. Let $I$ and $J$ be nonzero ideals in $\mathrm{Sk}_p$. Let $0\neq u\in I$ and let $0\neq v\in J$. By Lemma~\ref{lem5.25} we can assume that $u=x_1^{p-1}x_2^{p-1}\dots x_k^{p-1}$ and $v=x_1^{p-1}x_2^{p-1}\dots x_m^{p-1}$ for some $k,m>1$. By Lemma~\ref{lem5.2} there is an element $t=x_{k+1}^{i_{k+1}}\dots x_{q}^{i_q}\in J$ for some $q > k$. It is easy to see that $u\circ t\neq 0$, so $IJ\neq 0$.

Let us prove weak solvability. Let $f_1,\dots,f_k\in\mathrm{Sk}_p$. Then there is $m\in\mathbb{N}$ such that $f_1,\dots,f_k\in F[x_1,\dots,x_m]$. Let $S_n$ be the polynomial of the solvable identity. 

We have $S_1(x,y)=x\circ y$, so $S_1(x,1)=0$. Hence, a nonzero value of $S_n$ is possible only if the right argument in every bottom pair is not a constant. Hence at least the elements $f_{i_{2k}}$ are non-constants for $0<k\le 2^{n-1}$. Hence, $S_n(f_{i_1},\dots,f_{i_{2^n}})$ has a minimum degree of all its summands at least $2^{n-1}$. We have $S_n(f_{i_1},\dots,f_{i_{2^n}})\in F[x_1,\dots,x_{n+m}]$, hence some variable $x_i$ occurs with degree greater than $2^{n-1}/(m+n)$. We choose $n\in\mathbb{N}$ such that $2^{n-1}>p(m+n)$. Then $S_n(f_{i_1},\dots,f_{i_{2^n}})=0$ and $\mathrm{Sk}_p$ is a weakly solvable algebra. This proves the theorem.
\end{proof}

Let us denote the weakly solvable radical of Novikov algebras by $\mathcal{WS}$. Theorem~\ref{t5.3} and Corollary~\ref{c4.4} immediately imply the following result.

\begin{corollary}
In the variety of Novikov algebras over a field of characteristic $p>0$ we have a strict inclusion
\[\mathcal{B}\subsetneq \mathcal{K}\cap\mathcal{WS}.\]
\end{corollary}

\begin{proof}
The algebra $\mathrm{Sk}_p$ is r-nil (Theorem~\ref{t5.3}, item 3), hence $\mathrm{Sk}_p\in\mathcal{K}$. However, $\mathrm{Sk}_p$ is a prime algebra (Theorem~\ref{t5.3}, item 2), therefore $\mathcal{B}(\mathrm{Sk}_p)=0$, meaning $\mathrm{Sk}_p\notin\mathcal{B}$.
\end{proof}

\section{Conclusion}

We proved that the K\"othe radical exists in Novikov algebras. We showed that the class of weakly right nilpotent Novikov algebras coincides with the class of r-nil Novikov algebras. Moreover, we construct an example of a prime r-nil Novikov algebra over a field of positive characteristic, demonstrating that the K\"othe radical strictly contains the Baer radical. It is the main difference from commutative or Lie-solvable case \cite{P2026}, where the Baer radical and the K\"othe radical coincide.

It remains an open question whether the inclusion $\mathcal{B}\subseteq\mathcal{K}$ is strict over a field of characteristic $0$.

\section*{Acknowledgments}

This research is supported by the Russian Science Foundation (project 26-71-00028).

\noindent Alexander Panasenko \\
Sobolev Institute of Mathematics \\
Acad. Koptyug ave. 4, 630090 Novosibirsk, Russia \\
Novosibirsk State University \\
Pirogova str. 2, 630090 Novosibirsk, Russia \\
e-mail: a.panasenko@g.nsu.ru

\end{document}